\documentclass[onecolumn,amsthm]{autart}

\usepackage{amsfonts}
\usepackage{amsmath}
\usepackage{amsthm}
\usepackage{graphicx}

\usepackage[numbers,sort&compress,longnamesfirst]{natbib} 

\graphicspath{{graphics_included/}}

\def\lcplus{\textbf{ edited}}

\def\lcversion{2007.04.17.01}

\newcommand{\R}{\mathbb{R}}
\renewcommand{\S}{\mathbb{S}}
\newcommand{\N}{\mathbb{N}}
\newcommand{\bmat}[1]{\begin{bmatrix}#1\end{bmatrix}}
\newcommand{\norm}[1]{\lVert{#1}\rVert}
\newcommand{\eemph}[1]{\textbf{\textit{#1}}}
\newcommand{\eps}{\varepsilon}
\newcommand{\abs}[1]{\lvert{#1}\rvert}
\newcommand{\vertt}{\big\vert}
\newcommand{\rank}{\mathop{\mathrm{rank}}}
\newcommand{\diag}{\mathop{\mathrm{diag}}}
\newcommand{\range}{\mathop{\mathrm{range}}}

\let\bl\bigl
\let\bbl\Bigl
\let\br\bigr
\let\bbr\Bigr

\begin{document}

\begin{frontmatter}

\title{Positive Forms and Stability of Linear Time-Delay Systems}

\thanks[footnoteinfo]{Corresponding author M.~Peet.}

\author[mp]{Matthew M. Peet}\ead{matthew.peet@inria.fr},
\author[ap]{Antonis Papachristodoulou}\ead{antonis@eng.ox.ac.uk},
\author[sl]{Sanjay Lall}\ead{lall@stanford.edu}

\address[mp]{INRIA-Rocquencourt,Domaine de Voluceau, Rocquencourt BP105, 78153 Le Chesnay Cedex, France}
\address[ap]{Department of Engineering Science, University of Oxford, Parks Road, Oxford, OX1 3PJ, U.K.}
\address[sl]{Department of Aeronautics and Astronautics,  Stanford University, Stanford CA 94305-4035, U.S.A.}


\begin{abstract}
  We consider the problem of constructing Lyapunov functions for
  linear differential equations with delays. For such systems it is
  known that exponential stability implies the existence of a
  positive Lyapunov function which is quadratic on the space of continuous
  functions. We give an explicit parametrization of a sequence
  of finite-dimensional subsets of the cone of positive Lyapunov functions using
  positive semidefinite matrices. This allows stability analysis of
    linear time-delay systems to be formulated as a semidefinite program.
\end{abstract}

\end{frontmatter}


\section{Introduction}\label{sec:Introduction}

In this paper we present an approach to the construction of Lyapunov
functions for systems with time-delays. Specifically, we are
interested in systems of the form
\begin{equation*}
  \dot{x}(t)=\sum_{i=0}^k A_i x(t-h_i)
  \label{eqn:sys1}
\end{equation*}
where $x(t)\in\R^n$.  In the simplest case we are given the delays
$h_0,\dots,h_k$ and the matrices $A_0,\dots,A_k$ and we would like
to determine whether the system is stable.  For such systems it is
known that if the system is stable, then there exists a Lyapunov
function of the form
\begin{align*}
  V(\phi) =
  \int_{-h}^{0}
  \bmat{\phi(0) \\ \phi(s) }^T
  M(s)
  \bmat{\phi(0) \\  \phi(s) }
  \, ds
  +  \int_{-h}^{0}  \int_{-h}^{0}
  \phi(s)^T N(s,t) \phi(t) \, ds \, dt,
\end{align*}
where $M$ and $N$ are piecewise-continuous matrix-valued functions,
and $h=\max\{h_0,\dots,h_k\}$. Here $\phi:[-h,0]\rightarrow \R^n$ is
an element of the state space, which in this case is the space of
continuous functions mapping $[-h,0]$ to $\R^n$. The function $V$ is
thus a quadratic form on the state space.  The derivative is also
such a quadratic form, and the matrix-valued functions which define
it depend affinely on $M$ and $N$.

In this paper we develop an approach which uses semidefinite
programming to construct piecewise-continuous functions $M$ and $N$
such that the function $V$ is positive and its derivative is
negative. The functions we construct are piecewise-polynomial.
Roughly speaking, we show that
\[
  V_1(\phi) = \int_{-h}^{0}
  \bmat{\phi(0) \\ \phi(s) }^T
  M(s)
  \bmat{\phi(0) \\  \phi(s) }
  \, ds
\]
is positive for all $\phi$ if and only if there exists a piecewise
continuous matrix-valued function $T$ such that
\begin{align*}
  M(t)  + \bmat{ T(t) & 0 \\ 0 & 0 } & \geq 0
  \quad \text{for all }t
  \\
  \int_{-h}^0 T(t) \, dt &=0.
\end{align*}
That is, we convert positivity of the integral to \emph{pointwise}
positivity of $M$. This result is stated precisely
in~Theorem~\ref{thm:mainpos}. Pointwise positivity may then be
easily enforced, and in the case of positivity on the real line this
is equivalent to a sum-of-squares constraint. The constraint that
$T$ integrates to zero is a simple linear constraint on the
coefficients of $T$. Notice that the sufficient condition that
$M(s)$ be pointwise nonnegative is conservative, and as the
equivalence above shows it is easy to generate examples where $V_1$
is nonnegative even though $M(s)$ is not pointwise nonnegative.

We also give a necessary and sufficient condition for positivity of
\[
    \int_{-h}^0 \, \int_{-h}^0 \phi(s)^T \;N(s,t)\; \phi(t)\,ds\,dt.
\]
Roughly speaking, if $N$ is polynomial, then the given form is
positive if and only if there exists a positive semidefinite matrix
$Q \succeq 0$ such that $N(s,t)=Z(s)^T Q Z(t)$, where $Z$ is a
vector of monomials. This condition is expressed as a constraint on
the coefficients of $N$ and is stated in Theorem~\ref{thm:mainpos2}.
Note that pointwise positivity of $N$ is not sufficient for
positivity of the functional. The condition that the derivative of
the Lyapunov function be negative is similarly enforced.


\subsection{Background}

The use of Lyapunov functions on an infinite dimensional space to
analyze differential equations with delay originates with the work
of~\citet{krasovskii_1963}.  For linear systems, quadratic Lyapunov
functions were first considered by~\citet{repin_1966}.  The book
of~\citet{gu_2003} presents many useful results in this area, and
further references may be found there as well as
in~\citet{hale_1993,kolmanovskii_1999} and~\citet{niculescu_2001}.
The idea of using sum-of-squares polynomials together with
semidefinite programming to construct Lyapunov functions originates
in~\citet{parrilo_2000}.


\subsection{Notation}

Let $\N$ denote the set of nonnegative integers.  Let $\S^n$ be the
set of $n\times n$ real symmetric matrices, and for $X\in \S^n$ we
write $X \succeq 0$ to mean that $X$ is positive semidefinite.  For
two matrices $A,B$, we denote the Kronecker product by $A \otimes
B$. For $X$ any Banach space and $I\subset \R$ any interval, let
$\Omega(I,X)$ be the space of all functions
\[
\Omega(I,X)=\{\, f \colon I\rightarrow X \,\}
\]
and let $C(I,X)$ be the Banach space of bounded continuous functions
\[
C(I,X) =\{\, f: I \rightarrow X \ \vert\ f \text{ is continuous and
  bounded} \,\}
\]
equipped with the norm
\[
\norm{f} = \sup_{t\in I } \norm{f(t)}_X.
\]
We will omit the range space when it is clear from the context; for
example we write $C[a,b]$ to mean $C([a,b],X)$. A function is called
$C^n(I,X)$ if the $i^{th}$ derivative exists and is a continuous
function for $i=0,\ldots,n$. A function $f\in C[a,b]$ is called
piecewise continuous if there exists a finite number of points $a<
h_1< \dots < h_k< b$ such that $f$ is continuous at all $x\in
[a,b]\backslash \{ h_1,\dots,h_k\}$ and its right and left-hand
limits exist at $\{ h_1,\dots,h_k\}$.

Define also the projection $H_t:\Omega[-h,\infty) \rightarrow
\Omega[-h,0]$ for $t\geq 0$ and $h>0$ by
\[
(H_t x)(s) = x(t+s) \quad \text{for all }s\in[-h,0].
\]
We follow the usual convention and denote $H_t x$ by $x_t$.


\section{System Formulation}

Suppose $0=h_0 < h_1 < \dots < h_k=h$. Define the sets
$H=\{-h_0,\dots,-h_k\}$ and $H^c=[-h,0]\backslash H$ and suppose
$A_0,\dots,A_k\in\R^{n \times n}$.  We consider linear differential
equations with delay, of the form
\begin{equation}
  \label{eqn:sys}
  \dot{x}(t) = \sum_{i=0}^k A_i x(t-h_i)
  \quad \text{for all }t\geq 0,
\end{equation}
where the trajectory $x:[-h,\infty)\rightarrow \R^n$. The boundary
conditions are specified by a given function $\phi:[-h,0]\rightarrow
\R^n$ and the constraint
\begin{equation}
  \label{eqn:bc}
  x(t)=\phi(t) \quad \text{for all }t\in[-h,0].
\end{equation}
If $\phi\in C[-h,0]$, then there exists a unique function $x$
satisfying~(\ref{eqn:sys}) and~(\ref{eqn:bc}). The system is called
\eemph{exponentially stable} if there exists $\sigma >0$ and $a\in \R$
such that for every initial condition $\phi\in C[-h,0]$ the
corresponding solution $x$ satisfies
\[
\norm{x(t)} \leq a e^{-\sigma t} \norm{\phi} \quad \text{for all }t
\geq 0.
\]
We write the solution as an explicit function of the initial
conditions using the map $G:C[-h,0]\rightarrow \Omega[-h,\infty)$,
defined by
\[
(G\phi)(t) = x(t)  \quad \text{for all } t\geq -h,
\]
where $x$ is the unique solution of~(\ref{eqn:sys})
and~(\ref{eqn:bc}) corresponding to initial condition $\phi$. Also
for $s \geq 0 $ define the \emph{flow map}
$\Gamma_s:C[-h,0]\rightarrow C[-h,0]$ by
\[
\Gamma_s \phi = H_s G \phi,
\]
which maps the state of the system $x_t$ to the state at a later
time $x_{t+s} = \Gamma_s x_t$.


\subsection{Lyapunov Functions}

Suppose $V:C[-h,0]\rightarrow \R$. We use the notion of derivative
as follows. Define the \eemph{Lie derivative} of $V$ with respect to
$\Gamma$ by
\[
\dot{V}(\phi) = \limsup_{r\rightarrow 0^+}\frac{1}{r} \bl(
V(\Gamma_r \phi) - V(\phi) \br).
\]
We will use the notation $\dot{V}$ for both the Lie derivative and
the usual derivative, and state explicitly which we mean if it is
not clear from context.  We will consider the set $X$ of quadratic
functions, where $V\in X$ if there exist piecewise continuous
functions $M:[-h,0)\rightarrow \S^{2n}$ and $N:[-h,0)\times
[-h,0)\rightarrow \R^{n \times n}$ such that
\begin{align}
  \label{eqn:vclass}
  V(\phi) = \int_{-h}^{0}
  \bmat{\phi(0) \\ \phi(s) }^T
  M(s)
  \bmat{\phi(0) \\  \phi(s) }
  \, ds
+  \int_{-h}^{0}  \int_{-h}^{0} \phi(s)^T N(s,t) \phi(t) \, ds \,
dt.
\end{align}
The following important result shows that for linear systems with
delay, the system is exponentially stable if and only if there
exists a quadratic Lyapunov function.
\begin{thm}
  \label{thm:lyap}
  The linear system defined by equations~(\ref{eqn:sys})
  and~(\ref{eqn:bc}) is exponentially stable if and only if there
  exists a Lie-differentiable function $V\in X$ and $\eps>0$ such that
  for all $\phi\in C[-h,0]$
  \begin{equation}
    \label{eqn:lyapconds}
    \begin{aligned}
      V(\phi) &\geq \eps \norm{\phi(0)}^2 \\
      \dot{V}(\phi) & \leq - \eps \norm {\phi(0)}^2.
    \end{aligned}
  \end{equation}
  Further $V\in X$ may be chosen such that the corresponding functions
  $M$ and $N$ of equation~(\ref{eqn:vclass}) have the following
  smoothness property: $M(s)$ and $N(s,t)$ are bounded and continuous at all $s,t$
  such that $s\in H^c$ and $t\in H^c$.
\end{thm}
\begin{proof}
  See~\citet{kharitonov_2004} for a recent proof.
\end{proof}

\section{Positivity of Integrals}\label{sec:theory}

The goal of this section to present results which enable us to
computationally find functions $V\in X$ which satisfy the positivity
conditions in~(\ref{eqn:lyapconds}) and have the form
\begin{equation*}
      V(y)=
      \int_{-h}^0
      \bmat{y(0) \\ y(t) }^T
      M(t)
      \bmat{y(0) \\ y(t) }
      \, dt.
\end{equation*}
Before stating the main result in Theorem~\ref{thm:mainpos}, we give
a few preliminary lemmas.

\begin{lem}
  \label{lem:simple}
  Suppose $f\colon [-h,0]\rightarrow \R$ is piecewise continuous. Then the
  following are equivalent.
  \begin{itemize}
  \item[(i)]
    $\displaystyle
    \int_{-h}^0 f(t) \, dt \geq 0 $\\
  \item[(ii)] There exists a  function   $g \colon [-h,0]\rightarrow \R$
    which is piecewise continuous and satisfies
    \begin{align*}
      f(t) + g(t) & \geq 0 \quad \text{for all }t, \\
      \int_{-h}^0 g(t) \, dt &=0.
    \end{align*}
  \end{itemize}
\end{lem}
\begin{proof}
  The direction (ii)~$\implies$~(i) is immediate. To show the other direction,
  suppose (i)~holds, and let $g$ be
  \[
  g(t) = -f(t) + \frac{1}{h}\int_{-h}^0 f(s) \, ds \quad \text{for all
  }t.
  \]
  Then $g$ satisfies (ii).
\end{proof}

The second lemma shows that minimizing over continuous functions is
as good as minimizing over piecewise continuous functions.

\begin{lem}
  \label{lem:equalinfs}
  Suppose $H=\{-h_0,\dots,-h_k\}$ and let $H^c=[-h,0]\backslash H$. Let
  $f\colon [-h,0] \times \R^n \rightarrow \R$ be
  continuous on $H^c\times \R^n$, and suppose there exists a bounded
  function $z\colon [-h,0]\rightarrow \R$, continuous on $H^c$, such
  that for all~\hbox{$t\in[-h,0]$}
  \[
    f\bl(t,z(t)\br) = \inf_x f(t,x)
  \]
  Further suppose for each bounded set $X\subset \R^n$ the set
  \[
  \{\,
  f(t,x) \ \vert\
  x\in X, t\in[-h,0]
  \,\}
  \]
  is bounded. Then
  \begin{equation}
    \label{eqn:tp}
    \inf_{y\in C[-h,0]}
    \int_{-h}^{0}
    f\bl(t,y(t)\br)\, dt
    = \int_{-h}^0 \inf_x f(t,x)\, dt
  \end{equation}
\end{lem}
\begin{proof}
  Let
  \[
  K =  \int_{-h}^0 \inf_x f(t,x)\, dt
  \]
  It is easy to see that
  \[
  \inf_{y\in C[-h,0]}
  \int_{-h}^{0}
  f\bl(t,y(t)\br)\, dt
  \geq K
  \]
  since if not there would exist some continuous function $y$ and some
  interval on which
  \[
  f\bl(t,y(t)\br) < \inf_x f(t,x)
  \]
  which is clearly impossible.

  We now show that the left-hand side of~(\ref{eqn:tp}) is also less
  than or equal to $K$, and hence equals $K$.  We need to show that
  for any $\eps>0$ there exists $y\in C[-h,0]$ such that
  \[
  \int_{-h}^0 f\bl(t,y(t)\br) \, dt < K + \eps.
  \]
  To do this, for each $n\in \N$ define the set $H_n \subset\R$ by
  \[
  H_n = \bigcup_{i=1}^{k-1}
  ( h_i - \alpha/n, h_i+\alpha/n )
  \]
  and choose $\alpha>0$ sufficiently small so that $H_1\subset
  (-h,0)$.  Let $z$ be as in the hypothesis of the lemma, and pick $M$
  and $R$ so that
  \begin{align*}
    M &> \sup_{t\in [-h,0]} \norm{z(t)}\\
    R &= \sup
    \bl\{\,
    \abs{f(t,x)}
    \ \vertt\
    t\in[-h,0], \norm{x}\leq M
    \,\br\}.
  \end{align*}
  For each $n$ choose a continuous function $x_n:[-h,0]\rightarrow
  \R^n$ such that $x_n(t) = z(t)$ for all $t\not\in H_n$ and
  \[
  \sup_{t\in[-h,0]} \norm{x_n(t)} < M
  \]
  This is possible, for example, by linear interpolation.  Now we
  have, for the continuous function $x_n$
  \begin{align*}
    \int_{-h}^0 f\bl(t,x_n(t)\br) \, dt
    &= K + \int_{-h}^0
    \bbl(
    \hskip-1pt
    f\bl(t,x_n(t)\br) - f\bl(t,z(t) \br)
    \hskip-1pt
    \bbr) \, dt\\
    &= K + \int_{H_n}
    \bbl(
    \hskip-1pt
    f\bl(t,x_n(t)\br) - f\bl(t,z(t)\br)
    \hskip-1pt
    \bbr) \, dt\\
    & \leq K + 4R\alpha (k-1)/n
  \end{align*}
  This proves the desired result.
\end{proof}

The following lemma states that when the $\arg \min_z f(t,z)$ is
piecewise continuous in $t$, then we have the desired result.

\begin{lem}
  \label{lem:fz}
  Suppose $f\colon [-h,0] \times \R^n \rightarrow \R$ and the hypotheses of
  Lemma~\ref{lem:equalinfs} hold. Then the following are equivalent.
  \begin{itemize}
  \item[(i)]
    For all $y\in C[-h,0]$
    \[
    \int_{-h}^{0} f\bl(t,y(t)\br)\, dt \geq 0.
    \]
  \item[(ii)] There exists $g:[-h,0]\rightarrow \R$ which is piecewise
    continuous and satisfies
     \begin{align*}
      f(t,z) + g(t) & \geq 0 \quad \text{for all }t,  z\\
      \int_{-h}^0 g(t) \, dt &=0.
    \end{align*}
  \end{itemize}
\end{lem}
\begin{proof}
  Again we only need to show that (i) implies (ii). Suppose (i) holds,
  then
  \[
  \inf_{y\in C[-h,0]}
  \int_{-h}^{0}
  f\bl(t,y(t)\br)\, dt \geq 0
  \]
  and hence by Lemma~\ref{lem:equalinfs} we have
  \[
  \int_{-h}^0 r(t) \, dt \geq 0
  \]
  where $r:[-h,0]\rightarrow \R^n$  is given by
  \[
  r(t)=\inf_x f(t,x) \quad \text{for all }t.
  \]
  The function $r$ is continuous on $H^c$ since $f$ is continuous on
  $H^c\times \R^n$.  Hence by Lemma~\ref{lem:simple}, there exists
  $g$ such that condition (ii) holds, as
  desired.
\end{proof}

We now specialize the result of Lemma~\ref{lem:fz} to the case of
quadratic functions. It is shown that in this case, that under
certain conditions, the $\arg \min_z f(t,z)$ is piecewise
continuous.
\begin{thm}
  \label{thm:mainpos}
  Suppose $M:[-h,0]\rightarrow \S^{m+n}$ is piecewise continuous, and
  there exists $\eps>0$ such that for all~\hbox{$t\in[-h,0]$} we have
  \begin{align*}
    M_{22}(t) & \geq \eps I
  \end{align*}
where $M$ is partitioned as
\[
    M=\bmat{M_{11}&M_{12}\\M_{21}&M_{22}}
\]
with $M_{22}:[-h,0]\rightarrow \S^{n}$. Then the following are
equivalent.
  \begin{itemize}
  \item[(i)] For all $x\in\R^m$ and continuous $y:[-h,0]\rightarrow
    \R^n$
    \begin{equation}
      \label{eqn:xy}
      \int_{-h}^0
      \bmat{x \\ y(t) }^T
      M(t)
      \bmat{x \\ y(t) }
      \, dt \geq 0
    \end{equation}
  \item[(ii)] There exists a function $T:[-h,0]\rightarrow \S^m$ which
    is piecewise continuous and satisfies
    \begin{align*}
      M(t)  + \bmat{ T(t) & 0 \\ 0 & 0 } & \geq 0
      \quad \text{for all }t\in[-h,0]
      \\
      \int_{-h}^0 T(t) \, dt &=0
    \end{align*}
  \end{itemize}
\end{thm}
\begin{proof}
  Again we only need to show (i) implies (ii).  Suppose $x\in \R^n$,
  and define
  \[
  f(t,z) =
  \bmat{x \\ z}^T
  M(t)
  \bmat{x \\ z}
  \quad \text{for all }t,z.
  \]
  Since by the hypothesis $M_{22}$ has a lower bound, it is invertible
  for all $t$ and its inverse is piecewise continuous.  Therefore
  $z(t) = -M_{22}(t)^{-1} M_{21}(t) x$ is the unique minimizer of
  $f(t,z)$ with respect to $z$. By the hypothesis (i), we have that for all
  $y\in C[-h, 0]$
  \[
  \int_{-h}^0 f\bl(t,y(t)\br) \, dt  \geq 0.
  \]
  Hence by Lemma~\ref{lem:fz} there exists a function $g$ such that
  \begin{equation}
    \label{eqn:gfconds}
    \begin{aligned}
      g(t) + f(t,z) & \geq 0  \quad \text{for all } t,z \\
      \int_{-h}^0 g(t) \, dt &=0.
    \end{aligned}
  \end{equation}
  The proof of Lemma~\ref{lem:simple}  gives one such function as
  \[
  g(t) = - f\bl(t,z(t)\br)
  + \frac{1}{h}\int_{-h}^0
  f(s,z(s))
  \,dt.
  \]
  We have
  \[
  f\bl(t,z(t)\br)
  = x^T \bl(M_{11}(t) - M_{12}(t) M_{22}^{-1} (t) M_{21}(t) \br) x
  \]
  and therefore $g(t)$ is a quadratic function of $x$, say $ g(t) =
  x^T T(t) x$, and $T:[-h,0] \rightarrow \S^m$ is continuous on $H^c$.
  Then equation~(\ref{eqn:gfconds}) implies
  \[
  x^T T(t) x +  \bmat{x \\ z}^T M(t) \bmat{x \\ z} \geq 0
  \quad \text{for all }t,z,x
  \]
  as required.
\end{proof}

Notice that the strict positivity assumption on $M_{22}$ in
Theorem~\ref{thm:mainpos} is implied by the existence of an
$\epsilon > 0$ such that
\[
    V(x)\ge \epsilon \norm{x}_2^2,
\]
where $\norm{\cdot}_2$ denotes the $L_2$-norm.

We have now shown that the convex cone of functions $M$ such that
the first term of~\eqref{eqn:vclass} is nonnegative is exactly equal
to the sum of the cone of pointwise nonnegative functions and the
linear space of functions whose integral is zero. Note that
in~(\ref{eqn:xy}) the vectors $x$ and $y$ are allowed to vary
independently, whereas~(\ref{eqn:vclass}) requires that $x=y(0)$. It
is however straightforward to show that this additional constraint
does not change the result, using the technique in the proof of
Lemma~\ref{lem:equalinfs}.

The key benefit of this is that it is easy to parametrize the latter
class of functions, and in particular when $M$ is a polynomial these
constraints are semidefinite representable constraints on the
coefficients of $M$.


\section{Lie Derivatives}

In this section we will take the opportunity to define the
relationship between the functions $M$ and $N$ which define the
Lyapunov function $V$, and the functions $D$ and $E$, which define
the Lie derivative of the Lyapunov function $\dot{V}$.

\subsection{Single Delay Case}

We first present the single delay case, as it will illustrate the
formulation in the more complicated case of several delays.  Suppose
that $V\in X$ is given by~(\ref{eqn:vclass}), where
$M:[-h,0]\rightarrow \S^{2n}$ and $N:[-h,0]\times [-h,0]\rightarrow
\R^{n \times n}$. Since there is only one delay, if the system is
exponentially stable then there always exists a Lyapunov function of
this form with continuous functions $M$ and $N$.  Then the Lie
derivative of $V$ is
\begin{align*}
  \label{eqn:vdotcts}
  \dot{V}(\phi) = \int_{-h}^{0}
  \bmat{\phi(0) \\ \phi(-h) \\ \phi(s) }^T
  D(s)
  \bmat{\phi(0) \\ \phi(-h) \\   \phi(s) }
  \, ds
+  \int_{-h}^{0}  \int_{-h}^{0} \phi(s)^T E(s,t) \phi(t) \, ds \,
dt.
\end{align*}
Partition $D$ and $M$ as
\[
M(t) = \bmat{M_{11} & M_{12}(t) \\
  M_{21}(t) & M_{22}(t) }
\quad D(t) = \bmat{
  D_{11} & D_{12}(t)\\
  D_{21}(t) & D_{22}(t)}
\]
so that $M_{11}\in \S^{n}$ and $D_{11}\in \S^{2n}$.  Without loss of
generality we have assumed that $M_{11}$ and $D_{11}$ are constant.
The functions $D$ and $E$ are linearly related to $M$ and $N$ by
 \begin{align*}
    D_{11}  &=
    \bmat{A_0^T M_{11} + M_{11} A_0  &
      M_{11}A_1 \\[1mm]
      A_1^T M_{11}
      &  0}
    \\
    & + \frac{1}{h}
    \bmat{
      M_{12}(0) + M_{21}(0)    &
      - M_{12}(-h)   \\[1mm]
      - M_{21}(-h) & 0
    }
    \\
    &+\frac{1}{h}
    \bmat{ M_{22}(0)  &
      0 \\
      0 &
      - M_{22}(-h)
    }
    \\
    D_{12}(t) &=
    \bmat{
      A_0^T M_{12}(t) - \dot{M}_{12}(t) + N(0,t) \\[1mm]
      A_1^T M_{12}(t) - N(-h,t)
    }
    \\[2mm]
    D_{22}(t)  &= - \dot{M}_{22}(t)
    \\[2mm]
    E(s,t) &= \frac{\partial N(s,t)}{\partial s}  +
    \frac{\partial N (s,t)}{\partial t}.
  \end{align*}

\subsection{Multiple-delay case}

We now define the class of functions under consideration for the
Lyapunov functions. Define the intervals
\[
H_i = \begin{cases}
[-h_1,0] & \text{if }i=1\\
[-h_i,-h_{i-1}) & \text{if }i=2,\dots,k.
\end{cases}
\]
For the Lyapunov function $V$, define the sets of functions
\begin{align*}
  Y_1&=\bbl\{\,
  M:[-h,0] \rightarrow \S^{2n}
  \ \vert\ \\
  & \quad \qquad M_{11}(t) = M_{11}(s) && \hspace{-.5cm}\text{for all }s,t\in[-h,0] \\
  & \quad \qquad M \text{ is $C^1$ on $H_i$} && \hspace{-.5cm}\text{for all }i=1,\dots,k
  \  \,\bbr\}
  \\
  Y_2&=\bbl\{\,
  N:[-h,0]\times [-h,0] \rightarrow \S^{n}
  \ \vert\ \\
  & \quad \qquad N(s,t) = N(t,s)^T && \hspace{-.5cm}\text{for all }s,t\in[-h,0] \\
  & \quad \qquad N \text{ is $C^1$ on $H_i\times H_j$} &&  \hspace{-.5cm}\text{for all }
  i,j=1,\dots,k
  \ \bbr\}
\end{align*}
and for its derivative, define
\begin{align*}
  Z_1&=\bbl\{\,
  D:[-h,0] \rightarrow \S^{(k+2)n}
  \ \vert\ \\
  & \quad \qquad D_{ij}(t) = D_{ij}(s) && \hspace{-.5cm}\text{for all }s,t\in[-h,0] \\
  & && \text{for }i,j=1,\dots,3\\
  & \quad \qquad D \text{ is $C^0$ on $H_i$} && \hspace{-.5cm}\text{for all }i=1,\dots,k
  \  \,\bbr\}
  \\
  Z_2&=\bbl\{\,
  E:[-h,0]\times [-h,0] \rightarrow \S^{n}
  \ \vert\ \\
  & \quad \qquad E(s,t) = E(t,s)^T && \hspace{-.5cm}\text{for all }s,t\in[-h,0] \\
  & \quad \qquad E \text{ is $C^0$ on $H_i\times H_j$} &&  \hspace{-.5cm}\text{for all }
  i,j=1,\dots,k
  \ \bbr\}
\end{align*}
Here $M \in Y_1$ is partitioned according to
\begin{equation} \label{eqn:partm}
  M(t) = \bmat{
    M_{11}    &    M_{12}(t) \\
    M_{21}(t) &    M_{22}(t)
  },
\end{equation}
where $M_{11}\in \S^{n}$ and $D\in Z_1$ is partitioned according to
\begin{equation}
  \label{eqn:partd}
  D(t) = \bmat{
    D_{11}    & D_{12}    & D_{13} &    D_{14}(t) \\
    D_{21}    & D_{22}    & D_{23} &    D_{24}(t) \\
    D_{31}    & D_{32}    & D_{33} &    D_{34}(t) \\
    D_{41}(t) & D_{42}(t) & D_{43}(t) & D_{44}(t)
  }
\end{equation}
where $D_{11}, D_{33}, D_{44}\in\S^{n}$ and $D_{22}\in \S^{(k-1)n}$.
Let $Y=Y_1\times Y_2$ and $Z=Z_1 \times Z_2$. Notice that if $M\in
Y_1$, then $M$ need not be continuous at $h_i$ for $1\leq i \leq
k-1$, however, we require it be right continuous at these points. We
also define the derivative $\dot{M}(t)$ at these points to be the
right-hand derivative of $M$.  We define the continuity and
derivatives of functions in $Y_2, Z_1$ and $Z_2$ similarly.

We define the jump values of $M$ and $N$ at the discontinuities as
follows.
\[
\Delta M(h_i) = \lim_{t\rightarrow (-h_i)_+} M(t) -
\lim_{t\rightarrow (-h_i)_-} M(t)
\]
for each $i=1,\dots,k-1$, and similarly define
\[
\Delta N(h_i,t) = \lim_{s\rightarrow (-h_i)_+}N(s,t) -
\lim_{s\rightarrow (-h_i)_-}N(s,t)
\]
\begin{defn}
  Define the map $L: Y \rightarrow Z$ by $(D,E)=L(M,N)$ if for all
  $t,s\in [-h,0]$ we have
  \begin{align*}
    D_{11}  &=
    A_0^T M_{11} + M_{11}A_0 \\
     & \qquad + \frac{1}{h} \bl(
    M_{12}(0) + M_{21}(0) + M_{22}(0) \br)
    \\[1mm]
    D_{12}  &= \bmat{M_{11} A_1& \cdots & M_{11} A_{k-1}} \\
    & \qquad - \bmat{ \Delta M_{12}(h_1) & \cdots &
      \Delta M_{12}(h_{k-1})}
    \\[1mm]
    D_{13}  &= \frac{1}{h} \bl(
    M_{11} A_k - M_{12}(-h)
    \br)
     \\[1mm]
    D_{22}  & = \frac{1}{h} \diag\bl(-\Delta M_{22}(h_1),
    \dots, -\Delta M_{22}(h_{k-1})\br)
    \\[1mm]
    D_{23}  &= 0
    \\[1mm]
    D_{33}  &= -\frac{1}{h} M_{22}(-h)
    \\[1mm]
    D_{14}(t) & = N(0,t) + A_0^T M_{12}(t) - \dot{M}_{12}(t)
    \\[1mm]
    D_{24}(t) &= \bmat{ \Delta N (-h_1,t) + A_1^T M_{12}(t) \\
      \vdots\\
      \Delta N(-h_{k-1},t) + A_{k-1}^T M_{12}(t) }
    \\[1mm]
    D_{34}(t) &= A_k^T M_{12}(t) - N (-h,t)
    \\[1mm]
    D_{44}(t) &= -\dot{M}_{22}(t)
  \end{align*}
  and
  \[
  E(s,t) = \frac{\partial N(s,t)}{\partial s}  +
  \frac{\partial N (s,t)}{\partial t}.
  \]
  Here $M$ is partitioned as in~(\ref{eqn:partm}), $D$ is partitioned as in~\eqref{eqn:partd}, and the remaining
  entries are defined by symmetry.
\end{defn}

The map $L$ is the Lie derivative operator applied to the set of
functions specified by~(\ref{eqn:vclass}); this is stated precisely
below. Notice that this implies that $L$ is a linear map.

\begin{lem}
  \label{lem:derivative}
  Suppose $M\in Y_1$ and $N\in Y_2$ and $V$ is given by~(\ref{eqn:vclass}).
  Let $(D,E)=L(M,N)$.  Then the Lie derivative of $V$ is given by
  \begin{align}
    \label{eqn:vdot}
    \dot{V}(\phi) = \int_{-h}^{0}
    \bmat{ \phi(-h_0)   \\ \vdots \\  \phi(-h_k) \\ \phi(s) }^T
  D(s)
  \bmat{ \phi(-h_0)   \\ \vdots \\  \phi(-h_k) \\ \phi(s) }
  \, ds
+  \int_{-h}^{0}  \int_{-h}^{0} \phi(s)^T E(s,t) \phi(t) \, ds \,
dt.
\end{align}
\end{lem}
\begin{proof}
  The proof is straightforward by differentiation and integration
  by parts of~(\ref{eqn:vclass}).
\end{proof}


\section{Polynomial Matrices}

In this paper we use piecewise polynomial matrices as a conveniently
parametrized class of functions to represent the functions $M$ and $N$
defining the Lyapunov function~(\ref{eqn:vclass}) and its derivative.
Theorem~\ref{thm:mainpos} has reduced nonnegativity of the first term
of~(\ref{eqn:vclass}) to pointwise nonnegativity of a piecewise
polynomial matrix in one variable.

We first make some definitions which we will use in this paper; some
details on polynomial matrices may be found in~\citet{scherer_2004}
and~\citet{kojima_2003}.  We consider polynomials in $n$ variables.
As is standard, for $\alpha\in \N^n$ define the monomial in $n$
variables $x^\alpha$ by $ x^\alpha = x_1^{\alpha_1} \cdots
x_n^{\alpha_n}$.  We say $M$ is a real polynomial matrix in $n$
variables if for some finite set $W\subset \N^n$ we have
\[
M(x)= \sum_{\alpha \in W} A_\alpha x^\alpha
\]
where $A_\alpha$ is a real matrix for each $\alpha\in W$.  A
convenient representation of polynomial matrices is as a quadratic
function of monomials. Suppose $z$ is a vector of monomials in the
variables $x$, such as
\[
z(x) = \bmat{1 \\ x_1 \\ x_1 x_2^2 \\ x_3^4  }
\]
For convenience, assume the length of $z$ is $d+1$.  Let
$Q\in\S^{n(d+1)}$ be a symmetric matrix. Then the function $M$
defined by
\begin{equation}
  \label{eqn:quadmono}
  M(x) = (I_n \otimes z(x))^T Q (I_n \otimes z(x))
\end{equation}
is an $n\times n$ symmetric polynomial matrix, and every real
symmetric polynomial matrix may be represented in this way for some
monomial vector $z$.  If we partition $Q$ as
\[
Q=\bmat{Q_{11} &  \hdots & Q_{1n}\\
  \vdots&&\vdots \\
  Q_{n1} &  \hdots & Q_{nn}}.
\]
where each $Q_{ij}\in\R^{(d+1)\times (d+1)}$, then the $i,j$ entry of
$M$ is
\[
M_{ij}(x) = z(x)^T Q_{ij} z(x).
\]
Given a polynomial matrix $M$, it is called a \emph{sum of squares}
if there exists a vector of monomials $z$ and a positive
semidefinite matrix $Q$ such that equation~\eqref{eqn:quadmono}
holds. In this case,
\[
M(x) \succeq 0 \qquad \text{for all }x
\]
and therefore the existence of such a $Q$ is a sufficient condition
for the polynomial $M$ to be globally pointwise positive semidefinite.
A matrix polynomial in one variable is pointwise nonnegative
semidefinite on the real line if and only if it is a sum of squares;
see~\citet{choi_1980}. Given a matrix polynomial $M(x)$, we can test
whether it is a sum-of-squares by testing whether there is a matrix
$Q$ such that
\begin{align}
  \label{eqn:equatecoeffs}
  M(x) &=  (I_n \otimes z(x))^T Q (I_n \otimes z(x))\\
  Q  & \succeq 0,\notag
\end{align}
where $z$ is the vector of all monomials with degree half the degree
of $M$. Equation~(\ref{eqn:equatecoeffs}) is interpreted as equality
of polynomials, and equating their coefficients gives a finite set
of linear constraints on the matrix $Q$. Therefore to find such a
$Q$ we need to find a positive semidefinite matrix subject to linear
constraints, and this is therefore testable via semidefinite
programming. See~\citet{vandenberghe_1996} for background on
semidefinite programming.

\subsection{Piecewise polynomial matrices}

A matrix-valued function $M:[-h,0]\rightarrow \S^n$ is called a
\emph{piecewise polynomial matrix} if for each $i=1,\dots,k$ the
function $M$ restricted to the interval $H_i$ is a polynomial matrix.
We represent such piecewise polynomial matrices as follows.  Define
the vector of indicator functions \hbox{$g:[-h,0]\rightarrow \R^k$} by
\[
g_i(t) = \begin{cases}
  1 & \text{if } t \in H_i \\
  0 & \text{otherwise}
\end{cases}
\]
for all $i=1,\dots,k$ and all $t\in[-h,0]$. Let $z(t)$ be the vector
of monomials
\[
z(t) = \bmat{1 \\ t \\ t^2 \\ \vdots \\ t^d}
\]
and for convenience also define the function $Z_{n,d}:[-h,0]
\rightarrow \R^{{nk(d+1)}\times n}$ by
\[
Z_{n,d}(t) =  g(t)\otimes I_n \otimes z(t).
\]
Then it is straightforward to show that $M$ is a piecewise matrix
polynomial if and only if there exist matrices $Q_i \in \S^{n(d+1)}$
for $i=1,\ldots,k$ such that
\begin{equation}
  \label{eqn:piecequadmono}
  M(t) =
  Z_{n,d}(t)^T  \diag (Q_1,\dots,Q_k)
  Z_{n,d}(t).
\end{equation}
The function $M$ is pointwise positive semidefinite, i.e.,
\[
M(t) \succeq 0 \qquad \text{for all }t\in [-h,0]
\]
if there exists positive semidefinite matrices $Q_i$
satisfying~\eqref{eqn:piecequadmono}. We refer to such functions as
\emph{piecewise sum of squares matrices}, and define the set of such
functions
\begin{align*} &\Sigma_{n,d} = \bl\{\, Z_{n,d}^T(t) Q
Z_{n,d}(t) \ \vert\ \\
& \qquad \qquad Q=\diag (Q_1,\dots,Q_k),\, Q_i \in\S^{n(d+1)},
Q_i\succeq 0 \,\br\}.
\end{align*}
If we are given a function $M:[-h,0]\rightarrow \S^n$ which is
piecewise polynomial and want to know whether it is piecewise sum of
squares, then this is computationally checkable using semidefinite
programming. Naturally, the number of variables involved in this
task scales as $k n^2(d+1)^2$ when the degree of $M$ is $2d$.

\subsection{Piecewise Polynomial Kernels}

We consider functions $N$ of two variables $s,t$ which we will use as
a kernel in the quadratic form
\begin{equation}
  \label{eqn:qform}
  \int_{-h}^{0}  \int_{-h}^{0} \phi(s)^T N(s,t)\, \phi(t) \, ds \, dt
\end{equation}
which appears in the Lyapunov function~(\ref{eqn:vclass}). A
polynomial in two variables is referred to as a \emph{binary}
polynomial. A function $N:[-h,0] \times [-h,0] \rightarrow \S^n$ is
called a \emph{binary piecewise polynomial matrix} if for each
$i,j\in\{1,\dots,k\}$ the function $N$ restricted to the set $H_i
\times H_j$ is a binary polynomial matrix. It is straightforward to
show that $N$ is a symmetric binary piecewise polynomial matrix if
and only if there exists a matrix $Q\in\S^{nk(d+1)}$ such that
\[
N(s,t) =  Z_{n,d}^T(s) Q Z_{n,d}(t).
\]
Here $d$ is the degree of $N$ and recall
\[
Z_{n,d}(t) =  g(t) \otimes I_n \otimes z(t).
\]
We now proceed to characterize the binary piecewise polynomial
matrices $N$ for which the quadratic form~(\ref{eqn:qform}) is
nonnegative for all $\phi\in C([-h,0],\R^n)$.  We first state the
following Lemma.

\begin{lem}
  \label{lem:rank}
  Suppose $z$ is the vector of monomials
  \[
  z(t) = \bmat{1 \\ t \\ t^2 \\ \vdots \\ t^d}
  \]
  and the linear map $A:C[0,1]\rightarrow \R^{d+1}$ is given by
  \[
  A \phi = \int_0^1 z(t) \phi(t) \, dt
  \]
  Then $\rank A = d+1$.
\end{lem}
\begin{proof}
  Suppose for the sake of a contradiction that $\rank A < d+1$. Then
  $\range A $ is a strict subset of $\R^{d+1}$ and hence there exists
  a nonzero vector $q\in\R^{d+1}$ such that $q \perp \range A$. This
  means
  \[
  \int _0^1 q^T z(t) \phi(t) \, dt =0
  \]
  for all $\phi \in C[0,1]$. Since $q^Tz$ and $\phi$ are continuous
  functions, define the function $v:[0,1]\rightarrow \R$ by
  \[
  v(t) = \int_0^t q^T z(s) \, ds \qquad \text{for all }t\in[0,1].
  \]
  Since $v$ is absolutely continuous, we have for every $\phi \in
  C[0,1]$ that
  \begin{align*}
    \int_0^1 \phi(t) \, dv(t)
    &= \int _0^1 q^T z(t) \phi(t) \, dt
    \\ &=0
  \end{align*}
  where the integral on the left-hand-side of the above equation is
  the Stieltjes integral.  The function $v$ is also of bounded
  variation, since its derivative is bounded. The Riesz representation
  theorem~\cite{riesz_1990} implies that if $v$ is of bounded
  variation and
  \[
  \int _0^1  \phi(t) \, dv(t) =0
  \]
  for all $\phi \in C[0,1]$, then $v$ is constant on an everywhere
  dense subset of $(0,1)$. Since $v$ is continuous, we have $v$ is
  constant, and therefore $q^T z(t)=0$ for all $t$. Since $q^Tz$ is a
  polynomial, this contradicts the statement that $q\neq 0$.
\end{proof}

We now state the positivity result.

\begin{thm}\label{thm:mainpos2}
  Suppose $N$ is a symmetric binary piecewise polynomial matrix of degree $d$. Then
  \begin{equation}
    \label{eqn:qform2}
    \int_{-h}^{0}  \int_{-h}^{0} \phi(s)^T N(s,t) \phi(t) \, ds \, dt \geq 0
  \end{equation}
  for all $\phi\in C([-h,0],\R^n)$ if and only if there exists
  $Q\in\S^{nk(d+1)}$ such that
  \begin{align*}
    N(s,t) &= Z_{n,d}^T(s) Q Z_{n,d}(t) \\
    Q  & \succeq 0,
  \end{align*}
\end{thm}
\begin{proof}
  We only need to show the \emph{only if} direction. Suppose $N$ is a
  symmetric binary piecewise polynomial matrix. Let $d$ be the degree of $N$. Then there exists a
  symmetric matrix $Q$ such that
  \[
  N(s,t) = Z_{n,d}^T(s) Q Z_{n,d}(t).
  \]
  Now suppose that the inequality~(\ref{eqn:qform2}) is satisfied for
  all continuous functions $\phi$. We will show that every such $Q$ is
  positive semidefinite. To see this, define the linear map
  $J:C([-h,0],\R^n) \rightarrow \R^{nk(d+1)}$ by
  \[
  J\phi = \int_{-h}^0 (g(t) \otimes I_n \otimes z(t)) \phi(t) \, dt.
  \]
  Then
  \[
  \int_{-h}^{0}  \int_{-h}^{0} \phi(s)^T N(s,t) \phi(t) \, ds \, dt
  = (J \phi)^T Q (J\phi).
  \]
  The result we desire holds if $\rank J = nk(d+1)$, since in this
  case $\range J = \R^{nk(d+1)}$.  Then if $Q$ has a negative
  eigenvalue with corresponding eigenvector $q$, there exists $\phi$
  such that $q=J\phi$ so that the quadratic form will be negative,
  contradicting the hypothesis.

  To see that $\rank J = nk(d+1)$, define for each $i=1,\dots,k$ the
  linear map $L_i:C[H_i] \rightarrow \R^n$ by
  \[
  L_i \phi = \int_{H_i} z(t) \phi(t) \, dt
  \]
  Then if we choose coordinates for $\phi$ such that
   \[
  \phi = \bmat{
    \phi|_{H_1}  \\
    \phi|_{H_2} \\
    \vdots\\
    \phi|_{H_k}
  }
  \]

  where $\phi|_{H_j}$ restriction of $\phi$
  to the interval $H_j$, then we have in these coordinates that $J$ is
  \[
  J = \diag (L_1,\dots,L_k) \otimes I_n.
  \]
   Further, by Lemma~\ref{lem:rank} the maps $L_i$ each satisfy
$\rank
  L_i=d+1$. Therefore $\rank J = nk (d+1)$ as desired.

\end{proof}

The following corollary gives a tighter degree bound on the
representation of $N$.

\begin{cor}\label{cor:dbound2}
Let $N$ be a binary piecewise polynomial matrix of degree $2d$ which
is positive in the sense of Equation~\eqref{eqn:qform2}, then there
exists a $Q\in\S^{nk(d+1)}$ such that
  \begin{align*}
    N(s,t) &= Z_{n,d}^T(s) Q Z_{n,d}(t) \\
    Q  & \succeq 0,
  \end{align*}
\end{cor}
\begin{proof}
The binary representation used in Theorem~\ref{thm:mainpos2} had the
form
\begin{align*}
    N(s,t) &= Z_{n,2d}^T(s) P Z_{n,2d}(t) \\
    P  & \succeq 0,
\end{align*}
where $P\in\S^{nk(2d+1)}$. However, in any such a representation, it
is clear that $P_{ij,ij}=0$ for $i=d+2,\dots,2d+1$ and
$j=1,\dots,kn$. Therefore, since $P\succeq 0$ these rows and columns
are $0$ and can be removed. Define $Q$ to be the reduction of $P$.
$Z_{n,d}$ is the corresponding reduction of $Z_{n,2d}$. Then $Q \in
\S^{nk(d+1)}$, $Q \succeq 0$, and
\[
    N(s,t)=Z_{n,2d}^T(s) P Z_{n,2d}(t)=Z_{n,d}^T(s) Q Z_{n,d}(t).
\]
\end{proof}
For convenience, we define the set of symmetric binary piecewise
polynomial matrices which define positive quadratic forms by
\begin{align*}
  \Gamma_{n,d} &= \bl\{\, Z_{n,d}^T(s) Q Z_{n,d}(t) \ \vert\ Q
  \in\S^{nk(d+1)}, Q\succeq 0 \,\br\}.
\end{align*}

As for $\Sigma_{n,d}$, if we are given a binary piecewise polynomial
matrix $N:[-h,0]\times [-h,0]\rightarrow \S^n$ of degree $2d$ and
want to know whether it defines a positive quadratic form, then this
is computationally checkable using semidefinite programming. The
number of variables involved in this task scales as $(nk)^2(d+1)^2$.

\section{Stability Conditions}

\begin{thm}
  \label{thm:main}
  Suppose there exist $d\in\N$ and piecewise matrix polynomials
  $M,T,N,D,U,E$ such that
  \begin{align*}
    M + \bmat{T & 0 \\ 0 & 0 } &\in \Sigma_{2n,d} \\
    -D + \bmat{U & 0 \\ 0 & 0 } &\in \Sigma_{(k+2)n,d} \\
    N &\in \Gamma_{n,d} \\
    -E & \in \Gamma_{n,d} \\
    (D,E) &= L (M,N) \\
    \int_{-h}^{0} T(s)\, ds &=0 \\
    \int_{-h}^{0} U(s)\, ds &=0 \\
    M_{11}& \succ 0 \\
    D_{11}& \prec 0
  \end{align*}
  Then the system defined by
  equations~(\ref{eqn:sys}) and~(\ref{eqn:bc}) is exponentially stable.
\end{thm}

\begin{proof}
  Assume $M,T,N,D,U,E$ satisfy the above conditions, and define the
  function $V$ by (\ref{eqn:vclass}). Then Lemma~\ref{lem:derivative}
  implies that $\dot{V}$ is given by~\eqref{eqn:vdot}.  The function
  $V$ is the sum of two terms, each of which is nonnegative. The first
  is nonnegative by Theorem~\ref{thm:mainpos} and the second is
  nonnegative since $N\in\Gamma_{n,d}$. The same is true for $\dot{V}$.
  The strict positivity conditions of equations~\eqref{eqn:lyapconds}
  hold since $M_{11}\succ0$ and $-D_{11}\succ0$, and Theorem~\ref{thm:lyap}
  then implies stability.
\end{proof}

The feasibility conditions specified in Theorem~\ref{thm:main} are
semidefinite-representable. In particular the condition that a
piecewise polynomial matrix lie in $\Sigma$ is a set of linear and
positive semidefiniteness constraints on its coefficients.  Similarly,
the condition that $T$ and $U$ integrate to zero is simply a linear
equality constraint on its coefficients. Standard semidefinite
programming codes may therefore be used to efficiently find such
piecewise polynomial matrices. Most such codes will also return a dual
certificate of infeasibility if no such polynomials exist.

As in the Lyapunov analysis of nonlinear systems using
sum-of-squares polynomials, the set of candidate Lyapunov functions
is parametrized by the degree $d$. This allows one to search first
over polynomials of low degree, and increase the degree if that
search fails.

There are various natural extensions of this result. The first is to
the case of uncertain systems, where we would like to prove
stability for all matrices $A_i$ in some given semialgebraic set.
This is possible by extending Theorem~\ref{thm:main} to allow
Lyapunov functions which depend polynomially on unknown parameters.
A similar approach may be used to check stability for systems with
uncertain delays. Additionally, stability of systems with
distributed delays defined by polynomial kernels can be verified. It
is also straightforward to extend the class of Lyapunov functions,
since it is not necessary that each piece of the piecewise
sums-of-squares functions be nonnegative on the whole real line.  To
do this, one can use techniques for parameterizing polynomials
nonnegative on an interval; for example, every polynomial $p(x)=f(x)
- (x-1)(x-2)g(x)$ where $f$ and $g$ are sums of squares is
nonnegative on the interval $[1,2]$.

\section{Numerical Examples}

In this section we present the results of some example computations
using the approach described above. The computations were performed
using Matlab, together with the SOSTOOLS~\cite{prajna_2002} toolbox
and SeDuMi~\cite{sturm_1999} code for solving semidefinite
programming problems.

\subsection{Illustration.}
Consider the process of proving stability using the results of this
paper. The following system has well-known stability properties.
\begin{equation}
    \dot{x}(t)=-x(t-1)\label{eqn:sys2}
\end{equation}
A Matlab implementation of the algorithm in this paper has been
developed and is available online. This implementation returns the
following Lyapunov function for system~\eqref{eqn:sys2}. For
symmetric matrices, sub-diagonal elements are suppressed.
\begin{align*}
    V(x)&=\int_{-1}^0 \bmat{x(0)\\x(s)}^T M(s)
        \bmat{x(0)\\x(s)}\,ds
    +\int_{-1}^0\int_{-1}^0 x(s)^T R(s,t) x(t) \, ds \,
    dt
\end{align*}
where
\[
M(s)=\bmat{27.3&\;-16.8+2.74s\\ & 24.3+8.53 s}
\]
and
\[
R(s,t)=9.08.
\]

Positivity is proven using the function
\[
    t(s)=-.915+1.83s
\]
and the sum of squares functions
\[
    Q(s)=\bmat{13 &\;-3.3\\&\;12.2} \ge 0
\]
and
\[
    V(s)=Z(s)^T\, L \,Z(s),
\]
where
\[Z(s)=\bmat{\;1&s&0&0\\0&0&\;1&s\;}^T
\]
and
\[
    L=\bmat{28.215\;&5.585&\;-16.8&\;-1.973\\&13&1.413&-3.3\\&&24.3&10.365\\&&&12.2} \ge 0.
\]
This is because $-s(s+1)\ge 0$ for $s \in [-1,0]$ and
\[
    M(s)+\bmat{t(s)&0\\0&0}=-s(s+1)Q(s)+ V(s).
\]
Furthermore
\[
    R(s)= 9.08 \ge 0.
\]
Therefore, by Theorems~\ref{thm:main} and~\ref{thm:mainpos}, the
Lyapunov function is positive.

The derivative of the function is given by
\begin{align*}
    \dot{V}(x)&=\int_{-1}^0 \bmat{x(0)\\x(-1)\\x(s)}^T
        D(s)\bmat{x(0)\\x(-1)\\x(s)}\,ds
\end{align*}
where
\[
-D(s)=\bmat{9.3\quad&7.76&\;-6.34\\\quad
&15.77&\;-7.72+2.74s\\\quad&&\;8.53}.
\]
Negativity of the function is proven using the function
\[
    U(s)=\bmat{.0055+.011s& \;-.272-.544s\,\\& \;-.458-.916
    s\,},
\]
where
\[
    \int_{-1}^0U(s)\,ds=0,
\]
and the sum-of-squares functions
\[
    X(s)=\bmat{8.86 &1.90& \;-3.23\\  &11.54 &\;-3.71\\& & \;7.74}
\]
and
\[
    Y(s)=Z(s)^T\, L \,Z(s),
\]
where
\[Z(s)=\bmat{1&s&0&0&0&0\\0&0&\;1&s&0&0\\0&0&0&0&\;1&s}^T
\]
and
\[
    L=\bmat{9.3&4.43&7.76&.61896&-6.34&-1.0371\\
    &8.86&1.281&1.9&-2.1929&-3.23\\
    &&15.77&5.77&-7.72&.01171\\
    &&&11.54&-.98171&-3.71\\
    &&&&8.53&3.87\\
    &&&&&7.74} \ge 0.
\]
Negativity follows since
\[
    -D(s)+\bmat{U(s)&0\\0&0}=-s(s+1)X(s)+ Y(s).
\]
Therefore, by Theorem~\ref{thm:main}, the derivative of the Lyapunov
function is negative. Stability follows by Theorem~\ref{thm:lyap}.

\subsection{A single delay.}  We consider the following instance
of a system with  a single delay.
\[
\dot{x}(t) = \bmat{0 & 1 \\ -2 & 0.1} x(t) + \bmat{0& 0 \\ 1 & 0}
x(t-h)
\]
For a given $h$, we use semidefinite programming to search for a
Lyapunov function of degree $d$ that proves stability. Using a
bisection search over $h$, we determine the maximum and minimum $h$
for which the system may be shown to be stable. These are
shown below.
\[
\hbox{\begin{tabular}{c|c|c}
$d$ & $h_{\min}$ & $h_{\max}$\\
\hline
1 & .10017 & 1.6249\\
\hline
2 & .10017 & 1.7172\\
\hline
3 & .10017 & 1.71785
\end{tabular}}
\]
When the degree $d=3$, the bounds $h_\text{min}$ and $h_\text{max}$
are tight~\cite{gu_2003}. For comparison, we include here the bounds
obtained by Gu~\cite{gu_2003} using a piecewise linear Lyapunov
function with $n$ segments.
\[
\hbox{\begin{tabular}{c|c|c}
$n$& $h_{\min}$ & $h_{\max}$\\
\hline
1 & .1006 & 1.4272\\
\hline
2 & .1003 & 1.6921\\
\hline
3 & .1003 & 1.7161
\end{tabular}}
\]

\subsection{Multiple delays.}  Consider the system with two delays
below.
\[
\dot{x}(t) = \bmat{0 & 1 \\ -1 & \frac{1}{10}} x(t)
+ \bmat{0 & 0 \\ -1 & 0} x(t- \textstyle \frac{h}{2})
+ \bmat{0 & 0 \\ 1 & 0} x(t-h)
\]
As above, using a bisection search over $h$, we prove stability
for the range of $h$ below.
\[
\hbox{\begin{tabular}{c|c|c}
$d$ & $h_{\min}$ & $h_{\max}$\\
\hline
1 & .20247 & 1.354\\
\hline
2 & .20247 & 1.3722
\end{tabular}}
\]
Here for degree $d=2$, the bounds obtained are tight. Again we include
here bounds obtained by Gu~\cite{gu_2003} using a piecewise linear
Lyapunov function with $n$ segments.
\[
\hbox{\begin{tabular}{c|c|c}
$n$ & $h_{\min}$ & $h_{\max}$\\
\hline
1 & .204 & 1.35\\
\hline
2 & .203 & 1.372
\end{tabular}}
\]


\section{Summary}

In this paper we developed an approach for computing Lyapunov
functions for linear systems with delay. In general this is a
difficult computational problem, and certain specific classes of
this problem are known to be NP-hard~\cite{toker_1996}. However, the
set of Lyapunov functions is convex, and this enables us to
effectively test nonemptiness of a subset of this set. Specifically,
we parameterize a convex set of positive quadratic functions using
the set of polynomials as a basis, and the main results here are
Theorems~\ref{thm:mainpos} and~\ref{thm:mainpos2}.  Combining these
results with the well-known approach using sum-of-squares
polynomials allows one to use standard semidefinite programming
software to compute Lyapunov functions. This gives a nested sequence
of computable sufficient conditions for stability of linear delay
systems, indexed by the degree of the polynomial. In principle this
enables searching over increasing degrees to find a Lyapunov
function, although further work is needed to enhance existing
semidefinite programming codes to make this more efficient in
practice.

It is possible that Theorem~\ref{thm:mainpos} and its proof techniques
are applicable more widely, specifically to stability analysis of
nonlinear systems with delay, as well as to controller synthesis for
such systems. One specific extension that is possible is analysis of
delay systems with uncertain parameters, for which sufficient
conditions for existence of a Lyapunov function may be given using
convex relaxations. It is also possible to analyze stability of
nonlinear delay systems, in the case that the dynamics are defined by
polynomial delay-differential equations. Further extensions to allow
synthesis of stabilizing controllers are of interest, and may be
possible.

\bibliography{delay}

\end{document}